\documentclass[11pt]{article}

\usepackage{enumitem}
\usepackage{epsfig}
\usepackage[mathscr]{euscript}
\usepackage{csquotes}

\usepackage[linesnumbered, ruled,noend,noline]{algorithm2e}

\usepackage{float}
\usepackage{graphicx}
\usepackage{ stmaryrd }
\usepackage{multirow}
\usepackage{gensymb}
\usepackage{array}
\usepackage{listings}
\usepackage{amsthm}

\usepackage{amssymb}
\usepackage{graphicx,amsmath,calc}

\newtheorem{theorem}{Theorem}[section]
\newtheorem{lemma}[theorem]{Lemma}
\newtheorem{corollary}[theorem]{Corollary}
\newtheorem{definition}[theorem]{Definition}

\numberwithin{equation}{section}

\title{Intrinsic Sequentiality in $\mathbf{P}$: Causal Limits of Parallel Computation}
\author{Jing-Yuan Wei\thanks{Wide Area Grid-Forming (Shenzhen) Energy Co. Ltd.
3039 Bao'an North Road, Shenzhen, China. Email: weijingyuan@gmail.com}}
\date{}

\begin{document}
\maketitle

\begin{abstract}
We study a polynomial-time decision problem in which each input
encodes a depth-$N$ causal execution in which a single non-duplicable
token must traverse an ordered sequence of steps, revealing at most
$O(1)$ bits of routing information at each step.
The uncertainty in the problem lies in identifying the delivery path
through the relay network rather than in the final accept/reject outcome,
which is defined solely by completion of the prescribed execution.

A deterministic Turing machine executes the process in $\Theta(N)$ time.
Using information-theoretic tools - specifically cut-set bounds for relay
channels and Fano's inequality - we prove that any execution respecting
the causal constraints requires $\Omega(N)$ units of causal time,
thereby ruling out asymptotic parallel speedup.

We further show that no classical $\mathbf{NC}$ circuit family can
implement the process when circuit depth is interpreted as realizable
parallel time.
This identifies a class of polynomial-time problems with intrinsic causal
structure and highlights a gap between logical parallelism and causal
executability.
\end{abstract}


\section{Introduction}

Recent studies of parallel and distributed systems emphasize that
communication costs --- latency, bandwidth, and propagation delays --- often dominate
end-to-end performance~\cite{Czarnul2025}.
This motivates a broader principle: under realistic constraints of causality
and finite information rate, some tasks exhibit \emph{intrinsic sequentiality}
and resist the polylogarithmic speedups typically associated with parallel
computation.

Sequential dependency chains already appear in classical problems such as
\emph{pointer chasing}~\cite{KushilevitzNisan1997}, where a value is obtained by
repeatedly applying functions along a chain of length $N$.
Because each step depends on the previous one, the computation cannot be
shortcut by parallel evaluation, illustrating that some tasks possess an
inherent sequential structure independent of the amount of available hardware.

In this work we examine such sequential processes from a different perspective.
Following an information-theoretic view of computation developed in~\cite{wei2026},
we focus not on the evaluation of a final value, but on the
\emph{resolution of routing uncertainty along a causal chain}.
The central uncertainty in our model is not a hidden message to be decoded,
but the identity of the correct delivery path through a hierarchical relay
network, where each step reveals only limited information about that path.

A useful intuition is provided by global courier networks
(e.g., FedEx, UPS, DHL), where a parcel traverses a sequence of relay hubs
\[
\textsc{Sender} \;\to\;
\text{country} \;\to\;
\text{province} \;\to\;
\text{city} \;\to\;
\text{street} \;\to\;
\textsc{Receiver}.
\]
Although many hubs operate in parallel, any given parcel is held by exactly one
hub at a time and advances sequentially, hop by hop, through this hierarchy.
At each hub, only a small amount of new routing information becomes relevant:
which outgoing branch among finitely many choices is the next branch in the prescribed path.
Thus the delivery path is revealed gradually as the parcel progresses through
the network.

Inspired by this form of causal propagation, we formalize the
\emph{Hierarchical Temporal Relay} (HTR) problem.
An HTR instance specifies a routing-and-verification process in which a
unique, non-duplicable token must traverse $N$ ordered steps.
At each step the current relay reveals and forwards only $O(1)$ bits of routing
information before handing the token to the next level.
Correctness is defined solely by faithful completion of this prescribed causal
execution.

Viewed from this perspective, information propagation itself becomes a
computational resource.
If each hop has $d$ possible child branches, then the uncertainty in HTR is the
identity of the target leaf among $d^N$ possible destinations, with entropy
\[
H(M)=N\log d.
\]
Because each causal step reveals only bounded routing information, the process
forms a communication cascade in which information about the delivery path
accumulates gradually along the causal chain.

HTR therefore induces a communication process with
\emph{zero parallel capacity in time}: increasing spatial parallelism does not
increase the rate at which information about the delivery path can propagate
along the chain.
A deterministic Turing machine executes the process in $\Theta(N)$ time.
Using standard information-theoretic tools - specifically cut-set bounds for relay
channels and Fano's inequality - we show that any execution respecting the causal
constraints requires $\Omega(N)$ units of causal time.

This perspective also clarifies the difference from classical dependency-chain
problems.
While pointer-chasing formulations emphasize the difficulty of evaluating a
deeply composed function, the HTR formulation emphasizes the progressive
resolution of routing uncertainty under causal constraints.
The computation succeeds only when information identifying the delivery path
has propagated through the entire relay chain, making information propagation
itself the fundamental resource.

HTR therefore exposes a gap between \emph{logical parallelism} and
\emph{causal executability}.
While classical $\mathbf{NC}$ circuits evaluate predicates in polylogarithmic
depth under idealized assumptions of global aggregation, they cannot implement
HTR when circuit depth is interpreted as realizable parallel time.
This reflects not a weakness of $\mathbf{NC}$ as a logical class, but a
limitation of its abstraction when applied to tasks whose semantics require
causal execution.

Section~\ref{sec:htr} defines the HTR problem.
Section~\ref{sec:parallelism} establishes a linear lower bound under finite-rate
causal communication.
Section~\ref{sec:NC-vs-HTR} shows that $\mathbf{NC}$ circuits cannot implement
HTR.
Section~\ref{sec:causality-abstraction} discusses broader implications.
The paper concludes with related work and final remarks.

\section{A Polynomial-Time Decision Problem with Intrinsic Causality}
\label{sec:htr}

To formalize the causally constrained process described in the introduction, we
define a decision problem whose specification explicitly includes a hop-by-hop
execution structure.
Its defining feature is that causal order is not an implementation choice but an
inseparable part of the problem instance.
From an information-theoretic viewpoint, this makes HTR a communication process
whose execution has zero parallel capacity in time.

\paragraph{Semantically defined states.}
An intermediate state is \emph{semantically defined} if it is uniquely determined
by a causally valid prefix of the input and if its existence is necessary for the
correct interpretation of subsequent steps.
In HTR, token states and validation predicates become defined only after their
corresponding prefixes have been causally executed; they cannot be bypassed,
reordered, or inferred in advance without invalidating the computation.

\paragraph{Predicate versus process.}
A predicate specifies only a Boolean outcome as a function of the input.
A process, by contrast, specifies a required causal execution whose intermediate
states are semantically defined and must be respected for correctness.
Equivalently, a process may be viewed as a predicate together with an intrinsic
causal execution structure encoded in the instance itself.

\subsection{Hierarchical encodings carry causal meaning}

A useful real-world analogy is the SWIFT (BIC) banking code
\[
\textsc{payer} \;\to\;
\text{AAAA}\;\text{BB}\;\text{CC}\;\text{DDD}
\;\to\; \textsc{payee},
\]
whose components identify institution, country, location, and branch.
Although syntactically a flat string, such codes are resolved hierarchically:
later components have no operational meaning until earlier ones are determined.
The encoding therefore specifies not merely data, but an ordered resolution
protocol whose semantics depend on respecting a prescribed causal order.

In HTR, the causal structure is embedded directly in the instance
representation - motivated by hierarchical encodings such as SWIFT codes - in order
to isolate this phenomenon in a simple computational form.

\subsection{Hierarchical Temporal Relay (HTR)}

\paragraph{Handoff steps and causal time.}
We distinguish between a \emph{handoff step} and \emph{causal time}.
A handoff step consists of performing a local validation and,
upon success, transferring the token to the corresponding child.

Causal time is an abstract unit of parallel time during which information may
propagate subject to causal constraints.
Because the token is non-duplicable and no step may be skipped, at most one
handoff step can occur per unit of causal time.
Thus the canonical execution of HTR requires at least $N$ units of causal time.

\begin{definition}[Hierarchical Temporal Relay (HTR)]
Fix a constant branching factor $d \ge 2$ and let $T_{d,N}$ be a complete $d$-ary
tree of depth $N$ with $d^N$ leaves.
An instance of \textsf{HTR} is specified by a sender information $a_0$ and a
hierarchical address
\[
\ell = (a_0, a_1, \ldots, a_N),
\qquad
a_0 = 1,\quad
a_i \in \{0,1\}^{\lceil \log_2 d \rceil}
\ \text{for } 1 \le i \le N,
\]
which identifies a unique target leaf.
The canonical encoding of the instance is the concatenated string
\[
x = a_0 \,\|\, a_1 \,\|\, \cdots \,\|\, a_N,
\]
of total length $n = \Theta(N \log d)$.

Although $x$ is static, its semantics are causal.
A single non-duplicable token begins at the root of $T_{d,N}$ and evolves through
a sequence of handoff steps.
At each step $i$, only the current token-holder may act, performing a local
validation depending only on the current state $s_i$ and the routing
information $a_i$, which determines the next branch.
If the validation succeeds, the token is transferred to the corresponding child;
otherwise execution halts and rejects.

The decision outcome is defined only as the terminal state of this prescribed
$N$-step execution.
\end{definition}

\paragraph{Causal availability of routing information.}
Although the hierarchical address $\ell=(a_0,\ldots,a_N)$ is explicitly encoded
in the input string $x$, the information it contains is not operationally
available everywhere in the system simultaneously.
Routing decisions are made locally and causally: a relay learns the routing
information relevant to its step only when the token arrives.
Thus the execution can be viewed as progressively revealing the delivery path
along the causal chain.

\paragraph{Internal semantic structure.}
Let $S$ be a finite token-state space with $|S| = O(1)$.
The token carries the current execution state $s_i \in S$.
At step $i$, the token-holder performs a local validation depending only on
the current state $s_i$ and the routing information $a_i$, represented by a
predicate
\[
f_i : S \times \{0,1\}^{\lceil \log_2 d \rceil} \to \{0,1\},
\qquad 1 \le i \le N,
\]
computable in $O(1)$ time.

If $f_i(s_i,a_i)=1$, the token state is updated locally and the token is handed
to the child indexed by $a_i$, producing the next state $s_{i+1}$.
If $f_i(s_i,a_i)=0$, execution halts and rejects.
At the final step, the outcome is
\[
\text{accept if } f_N(s_N,a_N)=1,\qquad \text{reject otherwise}.
\]
No strict prefix of an execution with all prior validations successful suffices
to determine the outcome; if some earlier validation fails, the execution halts
immediately and rejects.

\paragraph{Relay-channel interpretation.}
The causal structure of \textsf{HTR} coincides with that of a serial
communication relay chain.
Let
\[
M=(a_1,\ldots,a_N)
\]
denote the hierarchical address of the target leaf.
Although $M$ is encoded in the input, its routing information becomes
operationally available only through successive token handoffs.

In this interpretation, the root initially possesses the routing information,
and each hop $i\to i{+}1$ reveals the next routing information $a_i$ needed to
select the correct child among $d$ possibilities.
Thus information identifying the delivery path propagates step by step
along the relay chain.

The non-duplicable token enforces half-duplex operation and precludes
parallel transmissions.
Correctness is defined only upon successful traversal of the entire chain,
so the routing information must propagate through every relay in order.
Consequently, the process has zero parallel capacity in time.

\paragraph{Relation to classical relay channels.}
This causal structure closely parallels the serial relay-channel model
studied in information theory.
In such models, a message originating at a source must propagate through
a sequence of intermediate relays before reaching the destination,
with each relay able to transmit only limited information per unit time.
Classical cut-set bounds therefore imply that the end-to-end transmission
rate cannot exceed the minimum capacity of the links in the chain.

The \textsf{HTR} process can be viewed as a discrete computational analogue
of this setting.
At each causal step, the current token-holder acts as a relay that reveals
at most $O(1)$ bits of routing information and forwards the token to the
next node.
Because the token is non-duplicable and the causal chain cannot be
bypassed, information about the delivery path can advance by at most
one hop per unit of causal time.

Consequently, the execution time of the process is governed by the same
information-propagation constraint that governs relay-channel capacity:
the routing information identifying the delivery path must traverse the
entire relay chain.
This viewpoint allows the causal-time lower bound for \textsf{HTR} to be
derived using standard information-theoretic arguments such as cut-set
bounds and Fano's inequality.

\paragraph{Decision version.}
The decision problem asks whether the prescribed HTR execution terminates in the
\emph{accept} state.

\paragraph{Sequential complexity.}
A deterministic Turing machine can simulate the HTR process by maintaining the
current token state and performing one local update per handoff step.

\begin{lemma}
$\textsf{HTR} \in \mathbf{P}$.
\end{lemma}

\begin{proof}
Let $n = \Theta(N \log d)$ be the input length.
On input $x = a_0 \| a_1 \| \cdots \| a_N$, a deterministic Turing machine parses
the blocks sequentially and simulates the prescribed execution.
It stores the current token state $s_i \in S$, which requires $O(1)$ space.
At step $i$, it evaluates $f_i(s_i,a_i)$ and, if this value is $1$, updates the
token state to $s_{i+1}$ and applies the transition specified by $a_i$.
Each handoff step requires $O(1)$ time, and the simulation performs at most $N$
steps.
Hence the total running time is $O(N) = \mathrm{poly}(n)$.
\end{proof}

\paragraph{Causally structured problems in $\mathbf{P}$.}
It is convenient to refer to problems such as \textsf{HTR} as
\emph{causally structured problems in $\mathbf{P}$}.
In such problems, instances encode a prescribed causal execution,
and correctness is defined by the terminal state of an $N$-step process
with causally defined intermediate states, rather than by evaluation
of a single static predicate.

We use the informal notation $\mathbf{P_{\mathrm{causal}}}$ to denote
this subclass of $\mathbf{P}$.
The purpose of this terminology is descriptive rather than
foundational: it highlights problems whose execution semantics
intrinsically constrain the rate of causal information propagation.
In these problems, the order of execution is fixed by the instance
specification itself and must be respected by any correct implementation,
rather than being an implementation choice of the computational model.

To avoid confusion with circuit-theoretic notions of depth, we distinguish
\emph{logical} from \emph{causal} sequentiality.
Circuit depth in $\mathbf{NC}$ reflects logical dependence in the
specification of a computation, whereas the sequentiality in HTR arises
from constraints on causal order, state evolution, and information
propagation during execution.
The relationship between these notions depends on the underlying execution
model and lies beyond the scope of this work.

\section{Parallelism Under Finite-Rate Causal Communication}
\label{sec:parallelism}

We analyze the HTR process in an abstract execution framework that enforces
\emph{hop-by-hop causality} and \emph{finite-rate information flow}.
No assumptions are made about any specific machine model, circuit formalism,
or computational architecture.
Causality is treated as an intrinsic property of the problem instance itself.

Let
\[
M=(a_1,\ldots,a_N)
\]
denote the hierarchical address of the target leaf.
We view $M$ as a random message drawn uniformly from the set
$\mathcal M=[d]^N$, so that
\[
|\mathcal M|=d^N,
\quad \text{ and entropy   }
H(M)=N\log d .
\]

Although the system may contain exponentially many agents,
the non-duplicable token constraint enforces that at any causal time
\emph{only the current token-holder} may transmit information across a boundary
$i \to i{+}1$.
Each causal time therefore permits \emph{at most one} such handoff.
Consequently, every hop behaves as a communication channel of finite capacity
\[
C_i = \log d + O(1)
\quad \text{bits per unit of causal time},
\]
which is $O(1)$ since the branching factor $d$ is assumed to be constant.

Initially, the hierarchical address of the target leaf $\ell$
(equivalently, the message $M$ encoded by that address)
is known only at the root.
All other agents are independent of $M$ at time $0$
and may acquire information about $\ell$
only through successive causal token handoffs.

In this interpretation, the hierarchical address $M$ represents the
uncertainty to be resolved during execution.
Each causal handoff reveals at most $O(1)$ bits about $M$,
so the computation progressively accumulates information about $M$
over causal time.

\paragraph{Causal information flow.}
Under the canonical hierarchical encoding,
any valid execution of HTR induces a strictly ordered cascade of hop boundaries
\[
0 \;\to\; 1 \;\to\; \cdots \;\to\; N,
\]
where crossing boundary $i \to i{+}1$ corresponds to advancing the token
from level $i$ to level $i{+}1$.
Information about $M$ must traverse these boundaries in order;
it cannot be rerouted, duplicated, or bypassed.

We index global system states by hop and causal time:
$X_i^{(t)}$ denotes the system state conditioned on the token
being at level $i$ after $t$ units of causal time.
Along any valid execution, the token advances by \emph{at most one} hop per time.
Thus information can cross boundary $i \to i{+}1$
only during a causal time in which that hop occurs.

Formally, whenever boundary $i \to i{+}1$ is crossed at time $t$,
the HTR model enforces the finite-capacity constraint
\[
I\!\left(M \,;\, X_{i+1}^{(t+1)} \mid X_{i+1}^{(t)}\right)
\;\le\; C_i .
\]
If the boundary is not crossed at time $t$,
then $X_{i+1}^{(t+1)} = X_{i+1}^{(t)}$
and the mutual information increment is zero.

\paragraph{End-to-end information bound.}
By the cut-set bound for relay channels
\cite{CoverElGamal1979,ElGamalKim2011}, the achievable end-to-end
communication rate in a serial relay chain is bounded by
$\min_{0 \le i < N} C_i$.
Consequently, within $T$ units of causal time the total mutual
information deliverable to level $N$ satisfies
\begin{equation} \label{eq:TandI}
I(M ; X_N^{(T)})
\;\le\;
T \cdot \min_{0 \le i < N} C_i .
\end{equation}

In the HTR problem, correctness is defined by faithful completion
of the prescribed $N$-step causal execution.
Under the uniform distribution on hierarchical addresses,
the terminal behavior depends on revealing
\[
H(M)=N\log d
\]
bits of routing information.
No strict prefix of an execution with all prior validations successful
suffices to determine the outcome.

Let $\hat M=\hat M(X_N^{(T)})$ be the receiver's estimate of $M$, and let
$P_e=\Pr[\hat M\neq M]=o(1)$.
Fano's inequality gives
\[
H(M\mid X_N^{(T)})
\;\le\;
h(P_e)+P_e\log(|\mathcal M|-1),
\]
where $h(\cdot)$ is the binary entropy function.
Since $|\mathcal M|=d^N$, this implies
\[
H(M\mid X_N^{(T)})=o(H(M)).
\]
Therefore
\begin{equation} \label{eq:IandH}
I(M;X_N^{(T)})
\;=\;
H(M)-H(M\mid X_N^{(T)})
\;\ge\;
(1-o(1))H(M).
\end{equation}

Combining \eqref{eq:TandI} and \eqref{eq:IandH} yields
\[
T \cdot \min_{0 \le i < N} C_i
\;\ge\;
(1-o(1))H(M).
\]

\begin{theorem}
\label{thm:lower-bound}
Any execution of the HTR process that respects
hop-by-hop causality,
finite per-hop communication capacity,
and the non-duplicable token constraint
must satisfy
\[
T
\;\ge\;
\frac{(1-o(1))H(M)}{\min_i C_i}
\;=\;
\Omega(N),
\]
even in the presence of polynomial or exponential spatial parallelism.
\end{theorem}

Equivalently, the HTR execution induces a communication process whose
information propagation rate is constant in causal time and cannot be
increased by spatial parallelism.

\begin{corollary}
\label{cor:no-physical-speedup}
In any execution model where one unit of parallel time equals one unit
of causal time and each hop can transmit at most $O(1)$ bits
per unit of causal time, the HTR process requires parallel time
$\Omega(N)$.
\end{corollary}

Corollary~\ref{cor:no-physical-speedup} reflects a familiar phenomenon:
some processes are intrinsically sequential.
Parallel resources may assist execution,
but they cannot compress a fixed causal timeline.
HTR exhibits exactly this structure:
a single, non-duplicable state must traverse
a prescribed causal chain of hops,
and no amount of parallelism can shorten that chain.

\section{NC Circuits as a Model of Parallel Time}
\label{sec:NC-vs-HTR}

We now examine how the lower bound of
Section~\ref{sec:parallelism} interacts with the classical notion of parallel
time embodied in $\mathbf{NC}$ circuits.
The central distinction is that $\mathbf{NC}$ measures \emph{logical depth}
under idealized, non-causal access assumptions,
whereas HTR specifies a \emph{causally ordered execution}
whose semantics depend on respecting a fixed sequence of state transitions.

A Boolean circuit is a static object computing a function
\(x \mapsto y\).
To apply circuit depth to a \emph{process} rather than a function,
one must interpret depth as execution time and impose a causal semantics
on inter-layer information flow.

\paragraph{Causal interpretation of circuit depth.}
We interpret a depth-$D$ Boolean circuit as a synchronous parallel
computation executing for $D$ units of \emph{causal time},
where circuit layer $t$ corresponds to causal time $t$.
Information flows only forward from layer $t$ to layer $t{+}1$.
Thus the circuit-time index $t$ tracks elapsed causal time, whereas the
boundary index $i$ tracks position along the HTR relay chain; under the
constraints below, one unit of causal time can advance the token across
at most one boundary.

To model HTR within this framework, we impose the following causal constraints:
(i) the token state is unique and non-duplicable;
(ii) in each unit of causal time the execution may advance the token across
\emph{at most one} hop boundary $i \to i{+}1$ (or halt);
and (iii) any update across boundary $i \to i{+}1$ occurs only via the token
crossing that boundary. The token's next state may depend only on its current
local state and on the routing information $a_i$ given in the input, and the
state transition during the crossing is restricted to what can be realized
within one causal time, i.e., within the finite per-time capacity bound $C_i$.

We say that a circuit family \emph{implements} the HTR process
if, under this interpretation, the layer-by-layer evolution of the circuit
corresponds to a valid HTR execution respecting hop-by-hop causality,
finite-rate communication, and the non-duplicable token constraint.

Under this interpretation, any depth-\(D\) circuit that implements HTR
induces a valid HTR execution lasting at most \(D\) units of causal time.
Consequently, the communication-limited lower bound established in
Section~\ref{sec:parallelism} applies directly.

\begin{theorem}
\label{thm:htr-depth}
In the standard Boolean circuit model (uniform families of
polynomial-size, constant-fan-in circuits),
no polylogarithmic-depth circuit family - and therefore no
$\mathbf{NC}$ circuit family - can implement the HTR process.
\end{theorem}

\begin{proof}
Under the causal interpretation of circuit depth,
each circuit layer corresponds to one unit of causal time.
Thus, a depth-\(D\) circuit family that implements HTR
induces an execution of the HTR process completing in at most
\(D\) units of causal time.

If a polylogarithmic-depth circuit family could implement HTR,
the resulting execution would have
\[
D = \mathrm{polylog}(N)
\]
units of causal time.

By Theorem~\ref{thm:lower-bound}, any execution respecting hop-by-hop
causality and finite per-hop communication capacity requires
$\Omega(N)$ units of causal time.
Since $\mathrm{polylog}(N) = o(N)$, this is impossible for sufficiently
large $N$.
\end{proof}

\section{Causality, Semantics, and Computational Abstraction}
\label{sec:causality-abstraction}

Classical complexity classes are defined in terms of logical input-output
relations and abstract machine models.
They deliberately idealize away many aspects of physical execution,
including causal time, bandwidth limits, and ordered state evolution.
Such idealizations are appropriate - and often essential - for reasoning
about computational power at the level of predicates and reductions.

The HTR model departs from this paradigm in a controlled but fundamental way.
Here the encoding of an instance does not merely represent data but specifies
a causally ordered execution in which intermediate states have semantic
significance.
Correctness therefore depends not only on the final input-output relation,
but on the existence of a valid causal trajectory realizing it.
This perspective invites a reinterpretation of classical complexity
abstractions in terms of causal realizability and information flow.
In the following subsections we examine this viewpoint through the role of
encoding and semantics, a causal reading of
$\mathbf{P}$, $\mathbf{NP}$, and $\mathbf{NC}$, and its relation to
sequential information-acquisition processes that arise under intrinsic
causal constraints.

\subsection{Encoding, Semantics, and Causal Execution}

Classical complexity theory treats encodings as neutral representations of data.
Inputs are considered equivalent up to polynomial-time reductions, and the
semantics of a problem are identified solely with its input-output relation.
This perspective is appropriate for predicate evaluation, where intermediate
computational states carry no independent meaning.

HTR departs from this paradigm in a principled way.
Here, the encoding itself specifies a required causal execution: the ordered
structure of the input determines which operations may occur and in what order.
Later steps of the computation are semantically defined only relative to the
successful completion of earlier ones.
If this causal structure is ignored, rearranged, or bypassed, the computation is
no longer correct.

Consequently, encoding in HTR is not a purely syntactic choice but part of the
semantic specification of the problem.
Any correct representation must expose the hierarchical order of the instance.

\subsection{A Causal Reading of P, NP, and NC}

The linear lower bound of Theorem~\ref{thm:lower-bound} shows that any execution
of HTR respecting hop-by-hop causality and finite communication capacity must
take $\Omega(N)$ units of causal time.
Polylogarithmic-depth parallelism is therefore impossible without abandoning
these constraints.

From this perspective, the classical complexity classes admit a causal reading:

\begin{itemize}

\item \textbf{P:}
Deterministic computation proceeds through explicit causal transitions.
HTR lies in $\mathbf{P}$ because its execution is faithfully realized by the
step-by-step transition semantics of a deterministic Turing machine.

\item \textbf{NP:}
Nondeterminism abstracts away the causal acquisition of information by allowing
a computation to guess a certificate instantaneously and then verify it
deterministically.
Any causal search required to obtain the witness is not represented within the
model, but is absorbed into the nondeterministic guessing step; the modeled
execution therefore consists only of the deterministic verification.

\item \textbf{NC:}
Circuit-based parallelism abstracts away causal propagation by collapsing depth
and permitting global aggregation of information in polylogarithmic time.

\end{itemize}

This interpretation concerns the abstraction of execution rather than the
formal relationships between complexity classes.
The standard definitions of $\mathbf{P}$, $\mathbf{NP}$, and $\mathbf{NC}$
remain unchanged.
It emphasizes only that HTR is a problem whose semantics require a
causally faithful realization, and therefore falls outside models of
computation that ignore ordered information propagation.

\subsection{Sequential Information Acquisition under Intrinsic Causal Constraints}

The causal structure of the HTR process resembles
sequential information-acquisition problems that arise in
breeding and selection programs, adaptive experiments,
sequential drug trials, and real-options decision processes.
In such settings an objective is reached through a sequence
of intermediate decisions whose outcomes become known only
progressively, so that the uncertainty lies in the evolution
of the decision path itself rather than only in the final outcome.

An intuitive analogy is searching for a book in a library whose shelves
are organized as a hierarchy of rooms.
At each level one must read a routing code that determines which room to
enter next, and only the code obtained at the current level reveals where
the book may be located at the next level.
Even if the total number of books is exponential, the search cannot be
parallelized, because the correct path must be discovered step by step.
This mirrors the HTR process, where each relay reveals only bounded routing
information and the token must advance sequentially through the hierarchy.

From an information-theoretic perspective this produces a cascade of
observations in which uncertainty about the underlying path is resolved
stage by stage.
If each stage reveals only limited information, then identifying a trajectory
with entropy $H(M)$ requires $\Omega(H(M))$ sequential stages.

The HTR model isolates this phenomenon in a minimal computational form.
The hierarchical address represents the unknown routing path through the
relay network, and each relay step reveals only bounded information.
The resulting $\Omega(N)$ causal time reflects a general principle:
when information about a target trajectory must be acquired through a
causally ordered sequence of bounded-capacity observations, the process
cannot be asymptotically parallelized.

\section{Related Work}

Several lines of research study limits of parallelism arising from
communication constraints and sequential dependence.
Distributed round-based models such as \textsc{LOCAL} and \textsc{CONGEST}
analyze the number of communication rounds required to solve global tasks
in networks: \textsc{CONGEST} bounds per-edge bandwidth (typically
$O(\log n)$ bits per round), while \textsc{LOCAL} allows unbounded
messages but restricts information propagation to one hop per round.
Similarly, VLSI area-time tradeoff results limit parallel speedup by
accounting for physical interconnect constraints and wiring congestion.
Recent work on physically realizable or geometry-aware circuit models
incorporates locality and communication constraints directly into the
circuit abstraction.

\paragraph{Causal diameter in distributed systems.}
In distributed computing, the notion of \emph{causal diameter}
characterizes the time required for information originating at one node
to influence another through a chain of causally ordered messages.
This concept appears in analyses of dynamic networks, consensus
protocols, and asynchronous systems, where information propagation is
limited by message-passing delays and causal dependencies.

The \textsf{HTR} model can be viewed as a computational abstraction of
this phenomenon.
It isolates a causal chain in which routing information must propagate
sequentially through a hierarchy of relays.
The resulting $\Omega(N)$ lower bound reflects the causal diameter
imposed by finite-rate information transmission.

The present work differs from these approaches in a key respect.
Rather than imposing causality and bandwidth limits as restrictions on
the computational model, the \textsf{HTR} problem treats the hop-by-hop
causal process as part of the \emph{instance semantics}.
Each instance encodes a prescribed execution in which routing
information must propagate sequentially along a causal path through the
relay chain.
This shifts the focus from computing a final output value to revealing
the execution trajectory itself.
Consequently, the sequentiality of the process can be analyzed using
information-theoretic tools, interpreting the execution as a relay
communication process with bounded per-step capacity.

\section{Concluding Remarks}

The Hierarchical Temporal Relay (HTR) model illustrates a polynomial-time
problem whose semantics require a prescribed causal execution.
A single non-duplicable token must traverse an $N$-step relay chain,
and each step can reveal only bounded routing information.
Consequently, information identifying the execution path can propagate
through the system only at a finite rate.
Using standard information-theoretic arguments, we show that any
execution respecting these causal constraints requires $\Omega(N)$
units of causal time.

This implies that polylogarithmic-depth circuit families cannot
implement the HTR process when circuit depth is interpreted as
realizable parallel time.
The limitation therefore arises not from computational hardness,
but from the causal structure of the task itself:
parallel hardware cannot compress a fixed causal chain of
information propagation.

More broadly, HTR highlights a class of polynomial-time problems in
which instances encode a required causal execution rather than a
static predicate evaluation.
For such problems, sequentiality reflects limits on information flow
through a causal process rather than limits on computational power.
This perspective shifts attention from uncertainty about the final
result to uncertainty about the causal path leading to it.
Resolving this trajectory requires information to propagate step by
step along the execution chain.

In this sense, the HTR viewpoint parallels the role of communication
complexity in classical theory.
While communication complexity isolates the information that must be
exchanged between parties, HTR isolates the information that must
propagate along a causal execution path leading to the computation.

Clarifying this distinction between logical parallelism and causal
executability may help guide the development of computational models
that better capture computation as it occurs in physical space-time.

\end{document}